%% file: Anchored_isoperimetric_profile_is_Lipschitz_continuous.tex
\pdfoutput=1
\documentclass[a4paper,10pt]{article}
\usepackage[latin1]{inputenc}  
\usepackage[english,greek,frenchb,french]{babel}
\usepackage{url,csquotes}
\usepackage[hidelinks,hyperfootnotes=false]{hyperref}
\RequirePackage{geometry}
\usepackage{float}
\usepackage{amsfonts,amssymb,enumerate}
\usepackage{amsmath}
\allowdisplaybreaks[1]
\usepackage{amsthm}
\usepackage{graphicx}
\usepackage{bbm}
\usepackage{listings}
\usepackage{titling}
\usepackage{dsfont}
\usepackage{color}

\usepackage[latin1]{inputenc}  
\usepackage[english,greek,frenchb,french]{babel}
\usepackage{url,csquotes}
\usepackage{color}

\newcommand{\E}{\mathbb{E}}
    \newcommand{\Prb}{\mathbb{P}}

		\newcommand{\sS}{\mathbb{S}}
	\DeclareMathOperator{\Diam}{Diam}

				\DeclareMathOperator{\e}{e}
				
    \newcommand{\sZ}{\mathbb{Z}}
    \newcommand{\sC}{\mathcal{C}}

		\newcommand{\cE}{\mathcal{E}}
			\newcommand{\cG}{\mathcal{G}}
		
	\newcommand{\cB}{\mathcal{B}}
		\newcommand{\cI}{\mathcal{I}}
		\newcommand{\cH}{\mathcal{H}}
				\newcommand{\cN}{\mathcal{N}}
						\newcommand{\cL}{\mathcal{L}}
	\newcommand{\sR}{\mathbb{R}}
	\newcommand{\sN}{\mathbb{N}}
	\newcommand{\vv}{\overrightarrow{v}}

				\DeclareMathOperator{\cyl}{cyl}

    \newcommand{\ind}{\mathds{1}}
 \theoremstyle{plain}  
\newtheorem{thm}{Theorem}
\newtheorem{prop}{Proposition}
\theoremstyle{plain}

\newtheorem{lem}{Lemma}

\newlength{\separationtitre}

\theoremstyle{remark}
\newtheorem{rk}{Remark}[section]

\date{}

\title{Anchored isoperimetric profile of the infinite cluster in supercritical bond percolation is Lipschitz continuous\thanks{Research was partially supported by the ANR project PPPP (ANR-16-CE40-0016)}} % \thanks is optional. Insert line breaks with \\

\author{Barbara Dembin\thanks{LPSM UMR 8001, Universit{\'e} Paris Diderot, Sorbonne Paris Cit{\'e}, CNRS, F-75013 Paris, France, \textsc{email:} bdembin@lpsm.paris} }%AUTHORS

\begin{document}

 \selectlanguage{english}
\maketitle
\textsc{Abstract:} We consider an i.i.d. supercritical bond percolation on $\sZ^d$, every edge is open with a probability $p>p_c(d)$, where $p_c(d)$ denotes the critical parameter for this percolation. We know that there exists almost surely a unique infinite open cluster $\sC_p$ \cite{Grimmett99}. We are interested in the regularity properties in $p$ of the anchored isoperimetric profile of the infinite cluster $\sC_p$.  For $d\geq 2$, we prove that the anchored isoperimetric profile defined in \cite{DembinCheeger} is Lipschitz continuous on all intervals $[p_0,p_1]\subset (p_c(d),1)$.\\


\textsc{Ams msc 2010:} primary 60K35, secondary 82B43 % Edit. Separate items with ;
%\AMSSUBJSECONDARY{FIXME:} % Optional, separate items with ;

\section{Introduction}
%Isoperimetric problems are among the oldest problems in mathematics. They consist in finding sets that maximize the volume given a constraint on the perimeter or equivalently that minimize the perimeter to volume ratio given a constraint on the volume. These problems can be formulated in the anisotropic case. Given a norm $\nu$ on $\sR^ d$ and $S$ a continuous subset of $\sR^d$, we define the tension exerted at a point $x$ in the boundary $\partial S$ of $S$ to be $\nu(n_S(x))n_S(x)$, where $n_S(x)$ is the exterior unit normal vector of $S$ at $x$. The quantity $\nu(n_S(x))$ corresponds to the intensity of the tension that is exerted at $x$. We define the surface energy of $S$ as the integral of the intensity of the surface tension $\nu(n_S(x))$ over the boundary $\partial S$. An anisotropic  isoperimetric problem consists in finding sets that minimize the surface energy to volume ratio given a constraint on the volume. To solve this problem, in \cite{wulff_cluster}, Wulff introduced through the Wulff construction a shape achieving the infimum. This shape is called the Wulff crystal, it corresponds to the unit ball for a norm built upon $\nu$. Later, Taylor proved in \cite{taylor1975} that this shape properly rescaled is the unique minimizer, up to translations and modifications on a null set, of the associated isoperimetric problem.

The study of isoperimetric problems in the discrete setting is more recent than in the continuous setting. In the continuous setting, we study the perimeter to volume ratio; in the context of graphs, the analogous problem is the study of the size of edge boundary to volume ratio. This can be encoded by the Cheeger constant. For a finite graph $\cG=(V(\cG),E(\cG))$, we define the edge boundary $\partial_\cG A$  of a subset $A$ of $V(\cG)$ as  $$\partial_\cG A =\Big\{\,e=\langle x,y\rangle \in E(\cG):x\in A,y\notin A \,\Big\}\,. $$ We denote by $|B|$ the cardinal of the finite set $B$. The isoperimetric constant of $\cG$, also called Cheeger constant, is defined as
$$\varphi_\cG=\min\left\{\,\frac{|\partial_\cG A|}{|A|}\,: \, A\subset V(\cG), 0<|A|\leq \frac{|V(\cG)|}{2}\,\right\}\,.$$
This constant was introduced by Cheeger in his thesis \cite{thesis:cheeger}  in order to obtain a lower bound for the smallest eigenvalue of the Laplacian. The isoperimetric constant of a graph gives information on its geometry.

Let $d\geq 2$. We consider an i.i.d. supercritical bond percolation on the graph $(\sZ^d,\E^d )$ having for vertices $\sZ^d$ and for edges $\E^d$ the set of pair of nearest neighbors in $\sZ^d$ for the Euclidean norm. Every edge $e\in\E^d$ is open with a probability $p>p_c(d)$, where $p_c(d)$ denotes the critical parameter for this percolation. We know that there exists almost surely a unique infinite open cluster $\sC_p$ \cite{Grimmett99}. In this paper, we want to study how the geometry of $\sC_p$ varies with $p$ through its Cheeger constant. However, if we minimize the isoperimetric ratio over all possible subgraphs of $\sC_p$ without any constraint on the size, one can prove that $\varphi_{\sC_p}=0$ almost surely. For that reason, we shall minimize the isoperimetric ratio over all possible subgraphs of $\sC_p$ given a constraint on the size. There are several ways to do it. We can for instance study the Cheeger constant of the graph $\sC_n=\sC_p \cap [-n,n]^d$ or of the largest connected component \smash{$\widetilde{\sC}_n$} of $\sC_n$ for $n\geq 1$. As we have $\varphi_{\sC_p}=0$ almost surely, the isoperimetric constants $\varphi_{\sC_n}$ and $\varphi_{\widetilde{\sC}_n}$ go to $0$ when $n$ goes to infinity. Roughly speaking, by analogy with the full lattice, we expect that subgraphs of \smash{$\widetilde{\sC}_n$} that minimize the isoperimetic ratio have an edge boundary size of order $n^{d-1}$ and a size of order $n^d$. 

In \cite{biskup2012isoperimetry}, Biskup, Louidor, Procaccia and Rosenthal defined a modified Cheeger constant $\widetilde {\varphi}_{\sC_n}$ and proved that $n\widetilde {\varphi}_{\sC_n}$ converges towards a deterministic constant in dimension $2$. In \cite{Gold2016}, Gold proved the same result in dimension $d\geq3$. Instead of considering the open edge boundary of subgraphs within $\sC_n$, they considered the open edge boundary within the whole infinite cluster $\sC_p$, this is more natural because $\sC_n$ has been artificially created by restricting $\sC_p$ to the box $[-n,n]^d$. They also added a stronger constraint on the size of subgraphs of $\sC_n$ to ensure that minimizers do not touch the boundary of the box $[-n,n]^d$.  Moreover, they proved that the subgraphs achieving the minimum, properly rescaled, converge towards a deterministic shape that is the Wulff crystal. Namely, it is the shape solving the continuous anisotropic isoperimetric problem associated with a norm $\beta_p$ corresponding to the surface tension in the percolation setting. The quantity $n\widetilde {\varphi}_{\sC_n}$ converges towards the solution of a continuous isoperimetric problem. 

This modified Cheeger constant was inspired by the anchored isoperimetric profile $\varphi_n(p)$. This is another way to define the Cheeger constant of $\sC_p$, that is more natural in the sense that we do not restrict minimizers to remain in the box $[-n,n]^d$. It is defined as follows:
$$ \varphi_n(p)=\min\left\{\,\frac{|\partial_{\sC_p} H|}{|H|}: 0\in H\subset \sC_p,\, \text{ H connected, } 0<|H|\leq n^d\,\right\}\,, $$
 where we condition on the event $\{0\in\sC_p\}$.
We say that $H$ is a valid subgraph if $0\in H\subset \sC_p$, $H$ is connected and $|H|\leq n^d$. We also define the open edge boundary of $H$ as:
$$\partial ^o H=\Big\{\,e\in\partial H,\text{ $e$ is open}\,\Big\}\,$$
where $\partial H$ is the edge boundary of $H$ in $(\sZ^d,\E^d)$.
Note that if $H\subset \sC_p$, then 
$\partial_{\sC_p} H=\partial^o H$. 

We need to introduce some definitions to be able to define properly a limit shape in dimension $d\geq 2$. In order to build a continuous limit shape, we shall define a continuous analogue of the cardinal of the open edge boundary. In fact, we will see that the cardinal of the open edge boundary may be interpreted in term of a surface tension $\cI$, in the following sense. Given a norm $\tau$ on $\sR^d$ and a subset $E$ of $\sR^d$ having a regular boundary, we define $\cI_\tau(E)$ as 
$$\cI_\tau (E)=\int_{\partial E}\tau(n_E(x))\cH^{d-1}(dx)\,,$$
where $\cH ^{d-1}$ denotes the Hausdorff measure in dimension $d-1$ and $n_E(x)$ is the normal unit exterior vector of $E$ at $x$.
The quantity $\cI_\tau(E)$ represents the surface tension of $E$ for the norm $\tau$. At the point $x$,  the tension has intensity $\tau(n_E(x))$ in the direction of $n_E(x)$. We denote by $\cL ^d$ the $d$-dimensional Lebesgue measure. We can associate with the norm $\tau$ the following isoperimetric problem:
$$\text{minimize $\frac{\cI_\tau(E)}{\cL^d(E)}$ subject to $\cL^d(E)\leq 1$}\,.$$
We use the Wulff construction to build a minimizer for this anisotropic isoperimetric problem (see \cite{wulff_cluster}). We define the set $\widehat{W}_\tau$ as
$$\widehat{W}_\tau=\bigcap_{v\in\sS^{d-1}}\left\{x\in\sR^d:\, x\cdot v\leq \tau(v)\right\}\,,$$
where $\cdot$ denotes the standard scalar product and $\sS^{d-1}$ is the unit sphere of $\sR^d$. Taylor proved in \cite{taylor1975} that the set $\widehat{W}_\tau$ properly rescaled is the unique minimizer, up to translations and modifications on a null set, of the associated isoperimetric problem. We need to build an appropriate norm $\beta_p$ for our problem that will be directly related to the cardinal of the open edge boundary. 
 
In \cite{DembinCheeger}, Dembin proves the existence of the limit of $n\varphi_n(p)$ and that it converges towards the solution of the continuous isoperimetric problem associated with the norm $\beta_p$. 
\begin{thm}\label{thmheart}
Let $d\geq 2$, $p>p_c(d)$ and let $\beta_p$ be the norm that will be properly defined in section \ref{s2}. Let $W_p$ be a dilate of the Wulff crystal $\widehat{W}_{\beta _p}$ for the norm $\beta_p$ such that $\cL^d(W_p)=1/{\theta_p}$ where $\theta_p=\Prb(0\in\sC_p)$. Then, conditionally on the event $\{0\in\sC_p\}$,
$$\lim_{n\rightarrow \infty} n\varphi_n(p)=\frac{\cI_p(W_p)}{\theta_p\cL^d(W_p)}=\cI_p(W_p)\text{ a.s..}$$
\end{thm}
\begin{rk} Actually, the same result holds when we condition on the event $\{0\in\sC_{p_0}\}$ for any $p_0\in(p_c(d), p]$.
\end{rk}
\noindent In this paper, we aim to study the regularity properties of the anchored isoperimetric profile. This was first studied by Garet, Marchand, Procaccia, Th{\'e}ret in \cite{GaretMarchandProcacciaTheret}, they proved that the modified Cheeger constant in dimension $2$ is continuous on $(p_c(2),1]$. The aim of this paper is the proof of the two following theorems. The first theorem asserts that the anchored isoperimetric profile is Lipschitz continuous on every compact interval $[p_0,p_1]\subset(p_c(d),1)$.
\begin{thm}[Regularity of the anchored isoperimetric profile] \label{Cheethmd}
Let $d\geq 2$. Let $ p_c (d) < p_0 < p_1 < 1$. There exits a positive constant $\nu$ depending only on $d$, $p_0$ and $p_1$, such that for all $p, q \in [p_0 , p_1 ]$, conditionally on the event $\{0\in\sC_{p_0}\}$,
$$\lim_{n\rightarrow\infty}n|\varphi_n(q)-\varphi_n(p)|\leq \nu|q-p|\,.$$
\end{thm}
\begin{rk} Actually, the Cheeger constant is also continuous at $1$, this is not a consequence of Theorem \ref{Cheethmd} but it comes from the fact that the map $p\rightarrow\beta_p$ is continuous on $(p_c(d),1]$. This result is a corollary of Theorem 4 in \cite{flowconstant}.
\end{rk}
\begin{rk}We did not manage to obtain here that the anchored isoperimetric profile is Lipschitz continuous on $[p_0,1]$ for a technical reason that is due to a coupling we use in the proof of Theorem \ref{Cheethmd}. However, this restriction is likely irrelevant. 
\end{rk}
\noindent The second theorem studies the Hausdorff distance between two Wulff crystals associated with norms $\beta_p$ and $\beta_q$. 
\begin{thm}[Regularity of the anchored isoperimetric profile]\label{Wulffthmd}
Let $d\geq 3$. Let $ p_c (d) < p_0 < p_1 < 1$. There exits a positive constant $\nu'$ depending only on $d$, $p_0$ and $p_1$, such that for all $p, q \in [p_0 , p_1 ]$,
$$d_\cH(\widehat{W}_{\beta _p},\widehat{W}_{\beta _q})\leq \nu'|q-p|\,,$$
where $d_\cH$  is the Hausdorff distance between non empty compact sets of $\sR^d$.
\end{thm}
The key element to prove these two theorems is to prove the regularity of the map $p\mapsto \beta_p$. We recall that it is already known that the map $p\rightarrow\beta_p$ is continuous on $(p_c(d),1]$. 
\begin{thm}[Regularity of the flow constant] \label{heartflow}
Let $p_c(d)<p_0<p_1<1$. There exists a positive constant $\kappa$ depending only on $d$, $p_0$ and $p_1$, such that for all $p\leq q$ in $[p_0,p_1]$,
$$\sup_{x\in\mathbb{S}^{d-1}}|\beta_p(x)- \beta_q(x)|\leq \kappa |q-p| \,.$$
\end{thm}
The proof of this theorem will strongly rely on an adaptation of the proof of Zhang in \cite{Zhang2017}.
\begin{rk} In this paper, we choose to work on the anchored isoperimetric profile instead of the modified Cheeger constant because the norm we use is the same for all dimensions $d\geq 2$. The existence of the modified Cheeger constant in dimension $2$ uses another norm specific to this dimension (see \cite{biskup2012isoperimetry}). In \cite{Gold2016}, Gold proved the existence of the modified Cheeger constant for $d\geq 3$ with the same norm $\beta_p$. Actually, we believe that his proof also holds in dimension $2$ up to using similar combinatorial arguments as in \cite{DembinCheeger}. Therefore, Theorem \ref{Cheethmd} may be shown for the modified Cheeger constant in dimension $d\geq 2$ using the same ingredients as in this paper.
\end{rk}

Here is the structure of the paper. In section \ref{s2}, we define the norm $\beta_p$. We prove that the map $p\mapsto \beta_p$ is Lipschitz continuous in section  \ref{s3}. Finally, we prove Theorems \ref{Cheethmd} and \ref{Wulffthmd} in section \ref{s4}.

\section[s2]{Definition of the norm $\beta_p$}\label{s2}
We introduce now many notations used for instance in  \cite{Rossignol2010} concerning flows through cylinders.
Let $A$ be a non-degenerate hyperrectangle, that is to say a rectangle of dimension $d-1$ in $\sR^d$. Let $\vv$ be one of the two unit vectors normal to $A$. Let $h>0$, we denote by $\cyl(A,h)$ the cylinder with base $A$ and height $2h$ defined by 
$$\cyl(A,h)=\{x+t\vv\,: \,  x\in A,\, t\in[-h,h]\}\,.$$
The set $\cyl(A,h)\setminus A$ has two connected components, denoted by $C_1(A,h)$ and $C_2(A,h)$. For $i=1,2$, we denote by $C'_i(A,h)$ the discrete boundary of $C_i(A,h)$ defined by 
$$C'_i(A,h)=\left\{x\in\sZ^d\cap C_i(A,h)\,:\,
\exists y \notin \cyl(A,h),\, \langle x,y \rangle\in\E^d  \right\}\,.$$
We say that the set of edges $E$ cuts $C'_1(A,h)$ from $C'_2(A,h)$ in $\cyl(A,h)$ if any path $\gamma$ from $C'_1(A,h)$ to $C'_2(A,h)$ in $\cyl(A,h)$ contains at least one edge of $E$. We call such a set a cutset. For any cutset $E$, let $|E|_{o,p}$ denote the number of $p$-open edges in $E$. We shall call it the $p$-capacity of $E$. Define
$$\tau_p(A,h)=\min\left\{ |E|_{o,p}: \,\text{ $E$ cuts $C'_1(A,h)$ from $C'_2(A,h)$ in $\cyl(A,h)$}\right\}\,.$$
Note that it is a random quantity as $|E|_{o,p}$ is random, and that the cutsets in this definition are anchored at the border of $A$. This quantity is related to the fact that graphs that achieve the infimum in the definition of $\varphi_n(p)$ try to minimize their open edge boundary. We refer to section 3 in \cite{DembinCheeger} for more detailed explanations on the construction of this norm $\beta_p$. To build a norm upon this quantity, we use the fact that the quantity $\tau_p(A,h)$ properly renormalized converges towards a deterministic constant when the size of the cylinder goes to infinity. The following proposition is a corollary of Proposition 3.5 in  \cite{Rossignol2010}.

\begin{prop}[Definition of the norm $\beta_p$]\label{normbeta}
Let $d\geq 2$, $p>p_c(d)$, $A$ be a non-degenerate hyperrectangle and $\vv$ one of the two unit vectors normal to $A$. Let $h$ be a height function such that $\lim_{n\rightarrow \infty }h(n)=\infty$. The limit 
$$\beta_p(\vv)=\lim_{n\rightarrow \infty } \dfrac{\E[\tau_p(nA,h(n))]}{\cH^{d-1}(nA)}$$
exists and is finite. Moreover, the limit is independent of $A$ and $h$ and the homogeneous extension of $\beta_p$ to $\sR^d$ is a norm.
\end{prop}
\noindent As the limit does not depend on $A$ and $h$, in the following for simplicity, we will take $h(n)=n$ and $A= S(\vv)$ where $S(\vv)$ is a square isometric to $[-1,1]^{d-1}\times\{0\}$ normal to $\vv$. We will denote by $B(n,\vv)$ the cube $\cyl(nS(\vv),n)$ and by $\tau_p(n,\vv)$ the quantity $\tau_p(nS(\vv),n)$. 

\section[s3]{Regularity of the map $p\mapsto\beta_p$}\label{s3}
Let $p_0>p_c(d)$ and let $q>p\geq p_0$.
Our strategy is the following, we easily get that $\beta_p\leq \beta_q$ by properly coupling the percolations of parameters $p_c(d)<p<q$. The second inequality requires more work.
We denote by $E_{n,p}$ the random cutset of minimal size that achieves the minimum in the definition of $\tau_p(n,\vv)$. By definition, as $E_{n,p}$ is a cutset, we can bound by above $\tau_q(n,\vv)$ by the number of edges in $E_{n,p}$ that are $q$-open, which we expect to be at most $\tau_p(n,\vv)+C(q-p)|E_{n,p}|$ where $C$ is a constant. We next need to get a control of $|E_{n,p}|$ which is uniform in $p\in[p_0,1]$ of the kind $c_dn^{d-1}$ where $c_d$ depends only on $p_0$. In \cite{Zhang2017}, Zhang obtained a control on the size of the smallest minimal cutset corresponding to maximal flows in general first passage percolation, but his control depends on the distribution $G$ of the variables $(t(e))_{e\in\E^d}$ associated with the edges. We only consider probability measures $G_p=p\delta_1+(1-p)\delta_0$ for $p>p_c(d)$, but we need to adapt Zhang's proof in this particular case to obtain a control that does not depend on $p$ anymore.
More precisely, let us denote by $\cN_{n,p}$ the total number of edges in $E_{n,p}$. We have the following control on $\cN_{n,p}$.

\begin{thm}[Adaptation of Theorem 2 in \cite{Zhang2017}]\label{Zhang}
Let $p_0>p_c(d)$. There exist constants $C_1$, $C_2$ and $\alpha$ that depend only on $d$ and $p_0$ such that for all $p\in[p_0,1]$, for all $n\in\sN^*$,
$$\Prb_p\left[\cN_{n,p}>\alpha n^{d-1}\right]\leq C_1\exp(-C_2 n^{d-1})\,.$$
\end{thm}
\begin{rk}
The proof is going to be simpler than the proof of Theorem 2 in \cite{Zhang2017}, because passage times in our context can take only values $0$ or $1$, \textit{i.e.}, to each edge we associate an i.i.d random variable of distribution $G_p=p\delta_1+(1-p)\delta_0$ whereas Zhang considers in \cite{Zhang2017} more general distributions. Our setting is equivalent to bond percolation of parameter $p$ by saying that an edge is closed if its passage time is $0$, and open if its passage time is $1$.
\end{rk}

Let us briefly explain the idea behind that theorem. Let $p\geq p_0$. We work on bond percolation of parameter $p$ (equivalently on first passage percolation with distribution $G_p=p\delta_1+(1-p)\delta_0$). We aim at bounding the size of the smallest minimal cutset that cuts the set $C'_1(nS(\vv),n)$ from $C'_2(nS(\vv),n)$ in $B(n,\vv)$. To do so we do a renormalization at a scale $t$ in order to build a "smooth" minimal cutset. The collection $(B_t(u))_{u\in\sZ^d}$ is a partition of $\sZ^d$ into boxes of size $t$ and $\bar{B}_t(u)=\bigcup _{v \overset{*}{\sim} u}B_t(u)$ where $v\overset{*}{\sim} u$ if $\|u-v\|_\infty=1$.
We will need the following Lemma that controls the probability that a $p$-atypical event occurs in a cube. We will prove this lemma after proving Theorem \ref{Zhang}
\begin{lem}[Uniform decay of the probability an atypical event occurs]\label{Grim}
Let $p_0>p_c(d)$. There exist positive constants $C_1(p_0)$ and $C_2(p_0)$ depending only on $p_0$ and $d$ such that for all $p\geq p_0$, for all $u\in\sZ^d$, for all $t\geq 1$, 
\begin{align}\label{decay}
\Prb\left[\text{a $p$-atypical event occurs in $B_t(u)$}\right]\leq C_1(p_0)\exp(-C_2(p_0)t)\,.
\end{align}
\end{lem}
We would like to highlight the fact that in Lemmas 6 and 7 in \cite{Zhang2017}, Zhang proves the same result but with constants $C_1$ and $C_2$ depending on $p$. Obtaining a decay that is uniform for $p\in[p_0,1]$ is the key element to adapt this proof and show that the constant $\alpha$ in the statement of the Theorem \ref{Zhang} does depend only on $p_0$ and $d$.

As the original proof is very technical, the adaptation of the proof is also technical.
\begin{proof}[Adaptation of the proof of Theorem $1$ in \cite{Zhang2017} to get Theorem \ref{Zhang} using Lemma \ref{Grim}]
We keep the same notations as in \cite{Zhang2017}. The following adaptation is not self-contained. Let $p_0>p_c(d)$ and $\vv\in\sS^{d-1}$.
In \cite{Zhang2017}, the author bounds the size of the smallest minimal cutset that cuts a given set of vertices $V$ from infinity. However, his construction of a linear cutset in section $2$ of \cite{Zhang2017} is not specific to the set $B(k,m)$ and can be defined in the same way for any set of vertices. In particular we can replace $B(k,m)$ by $C'_1(nS(\vv),n)$ and $\infty$ by $C'_2(nS(\vv),n)$ (as it is done by Zhang in Theorem 2 in \cite{Zhang2017}). We denote by $\sC(n)$ the set that corresponds to $C(k,m)$ defined in Lemma 1 in \cite{Zhang2017}:
$$\sC(n)=\{v\in\sZ^d:\text{ $v$ is connected to $C'_1(nS(\vv),n)$ by an open path }\}\,.$$
We denote by $\cG(n)$ the event that $\sC(n)\cap C'_2(nS(\vv),n)=\emptyset$ (it corresponds to $\cG(k,m)$ in \cite{Zhang2017}). On this event, the exterior edge boundary $\Delta_e \sC(n)$ of $\sC(n)$ is a closed cutset that cuts $C'_1(nS(\vv),n)$ from $C'_2(nS(\vv),n)$.  We denote by $\underline{A}$ the set of $t$-cubes that intersect $\Delta_e \sC(n)$. By Zhang construction, we can extract from $\underline{A}$ a set of cubes $\Gamma_t$ such that $\Gamma_t$ is $*$-connected and the union $\bar{\Gamma}_t$ of the $3t$-cubes in $\Gamma_t$ (the cubes in $\Gamma_t$ and their $*$-neighbors) contains a cutset of null capacity that cuts the set $C'_1(nS(\vv),n)$ from $C'_2(nS(\vv),n)$. Moreover, each cube in $\Gamma_t$ has a $*$-neighbor where a $p$-atypical event occurs.

 As we only focus on edges inside $B(n,\vv)$, we can assume that all other edges are closed. Thus, the set $\Delta_e \sC(n)\setminus B(n,\vv)$ is included in the exterior edge boundary $\Delta_e B(n,\vv)$ of $B(n,\vv)$. Therefore, the cubes $B_t(u)$ in $\underline{A}$ such that $\bar{B}_t(u)$ is not contained in the strict interior of $B(n,\vv)$ satisfy $\bar{B}_t(u)\cap\Delta_eB(n,\vv)\neq\emptyset$. We deduce that there are at most $C_{d,t} n ^{d-1}$ such cubes in $\underline{A}$ (and so, in $\Gamma_t$) where $C_{d,t}$ is a constant depending only on the dimension $d$ and $t$. Moreover, any cube $B_t(u)$ that intersects the boundary $\Delta_eC'_1(nS(\vv),n)\setminus B(n,\vv)$ belongs to $\underline{A}$ as it also intersects $\Delta_e\sC(n)$ and by Zhang construction, we can prove that the cube $B_t(u)$ also belongs to $\Gamma_t$. Thanks to this remark, we avoid the part of Zhang's proof where he tries to find a vertex $z$ in the intersection between the cutset $W(k,m)$ and a line $L$ in order to find a cube that is in $\Gamma_t$. Thus, the term $\exp(\beta^{-1}n)$ in (6.19) is not necessary.

The set $E=\{\langle x,y\rangle\in B(n,\vv)\,:\,x\in C'_1(nS(\vv),n)\,\}$ cuts the set $C'_1(nS(\vv),n)$ from the set $C'_2(nS(\vv),n)$ in $B(n,\vv)$ and there exists a constant $c_d$ depending only on $d$ but not on $\vv$ such that $|E|\leq c_d n^{d-1}$. Thus, we obtain that $$\tau_p(n,\vv)\leq |E| \leq c_d n^{d-1} \,.$$ We denote by $E_{n,p}$ the cutset that achieves the infimum in $\tau_p(n,\vv)$ and such that $|E_{n,p}|=\cN_{n,p}$ ($E_{n,p}$ corresponds to $W(k,m)$ in \cite{Zhang2017}). For a configuration $\omega$, we denote by $e_1,\dots,e_{J(\omega)}$ the $p$-open edges in $E_{n,p}$. We have $J(\omega)=\tau_p(n,\vv)(\omega) \leq c_d n^{d-1} $. We denote by $\sigma(\omega)$ the configuration which coincides with $\omega$ except in edges $e_1,\dots,e_{J(\omega)}$ that are closed for $\sigma(\omega)$. Thus, the set $E_{n,p}(\sigma(\omega))$ is a $p$-closed (for the configuration $\sigma(\omega)$) cutset that cuts $C'_1(nS(\vv),n)$ from $C'_2(nS(\vv),n)$ in $B(n,\vv)$. Note that the set of edges $E_{n,p}(\sigma(\omega))$ is determined by the configuration $\omega$ whereas we consider its capacity for $\sigma(\omega)$. We recall that all the edges outside $B(n,\vv)$ are closed so that the event $\cG(n)$ occurs in the configuration $\sigma(\omega)$ and we can use the construction of section 2 in \cite{Zhang2017}: $\bar{\Gamma}_t$ contains a $p$-closed (for $\sigma(\omega)$) cutset $\textbf{$\Gamma$}$ that cuts $C'_1(nS(\vv),n)$ from $C'_2(nS(\vv),n)$ (see Lemma 4 in \cite{Zhang2017}). By taking the intersection of this cutset with the box $B(n,\vv)$, we obtain the existence of a closed cutset that cuts  $C'_1(nS(\vv),n)$ from $C'_2(nS(\vv),n)$ in $B(n,\vv)$. 
%Thus $\sC(n)\subset B(n,\vv)$$, the event $\cG(n)$ occurs and we can use the construction of section 2: $\bar{\Gamma}_t$ contains a closed cutset $\textbf{$\Gamma$}$. Note that $\partial_e\sC(n)$ cannot be outside the boundary of $ B(n,\vv)$, so we may choose $\Gamma$ such that it only uses edges in the boundary and in the interior of $\cyl$.
%Let $\Gamma'=\Gamma\cap B(n,\vv)$ and $\Gamma'_t=\{B_t(u)\in\Gamma_t\,:\, B_t(u)\cap\Gamma'\neq\emptyset\}$.
%Let us show that $\Gamma'(\omega))$ is a closed set for the configuration $\sigma(\omega)$ inside $\Gamma_t'(\sigma(\omega))$ that also cut $C'_1$ from $C'_2$. Note that $\partial_e\sC(n)$ only used edges in the boundary and the edges inside $ B(n,\vv)$. If  $\partial_e\sC(n)$ uses a surface edge, then there exists an open path from $C'_1$ to $C'_2$ inside $ B(n,\vv)$  and it contradicts the fact that $E_{n,p}(\sigma(\omega))$ is a zero cutset. As $\partial \sC(n)\subset B(n,\vv)$, by Lemma 5  in \cite{Zhang2017}, $\Gamma(\sigma(\omega))\subset B(n,\vv)$ and any path from $C'_1$ to $C'_2$ has to use a closed edge of $\Gamma'(\sigma(\omega))$.

We now change $\sigma(\omega)$ back to $\omega$. For $i\in\{1,\dots,J(\omega)\}$, the passage time of $e_i$ changes from $0$ to $1$. We write $\Gamma(\omega)$ when we consider the edge set $\Gamma$ with its edges capacities determined by the configuration $\omega$. The set $\Gamma(\omega)$ exists as an edge set, it is still a cutset but it is no longer closed, all edges in $\Gamma(\omega)$ except the $e_i$ are closed. 
Therefore, $|\Gamma(\omega)|_{o,p}\leq J(\omega)$, but by definition of $E_{n,p}$, we have  $J(\omega)=|E_{n,p}(\omega)|_{o,p}\leq |\Gamma(\omega)|_{o,p}\leq J(\omega)$ and so $|\Gamma(\omega)|_{o,p}=J(\omega)$.
Moreover, for each $\omega$, by definition of $\cN_{n,p}(\omega)$, we get that $|\Gamma(\omega)|\geq \cN_{n,p}(\omega)$. 

Note that for the $t$-cubes $B_t(u)\in\Gamma_t$ such that $\bar{B}_t(u)$ intersects the boundary of $B(n,\vv)$, we cannot be sure that there exists a $t$-cube in $\bar{B}_t(u)$ where a $p$-atypical event occurs, but the number of such cubes is at most $C_{d,t}n^{d-1}$. Thus, if the number of $t$-cubes in $\Gamma_t$ is greater than $\beta n^{d-1}$, then the number of $t$-cubes in $\Gamma_t$ that do not intersect the boundary of $B(n,\vv)$ and that do not contain any edge among $e_1,\dots,e_J$ is greater than $(\beta- C_{d,t}-c_d)n^{d-1}$. All these $t$-cubes have at least one $*$-neighbor with a blocked or disjoint property. This leads to small modifications of constants in the proof of \cite{Zhang2017}. We insist on the fact that the remainder of the proof is the same except that we use Lemma \ref{Grim}, \textit{i.e.}, a uniform decay for $p\in[p_0,1]$ of the probability of a $p$-atypical event instead of using the control in \cite{Zhang2017}.
\end{proof}
Let us now prove Lemma \ref{Grim}. We need to adapt some existing proofs in order to obtain a decay which is uniform in $p$. Let us first introduce some useful definitions.

\begin{figure}[ht]
\def\svgwidth{0.8\textwidth}
\begin{center}
 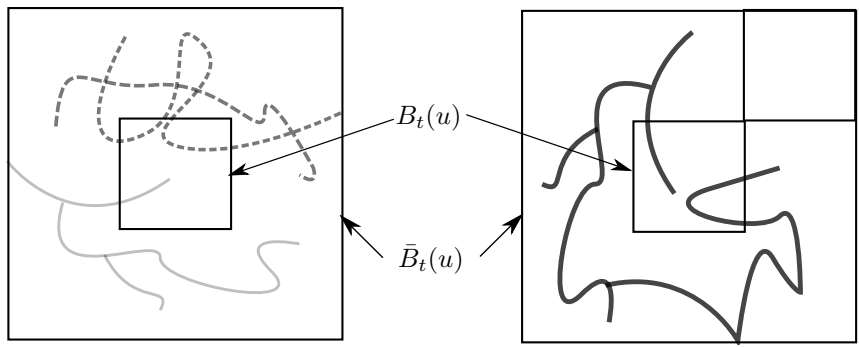
 \vspace{-1cm}
 \caption{On the left a box with a disjoint property, on the right a box with a blocked property}
 \label{atypical}
 \end{center}
\end{figure}
 
A connected cluster $C$ is said to be $p$-crossing for a box $B$, if for all $d$ directions, there is a $p$-open path in $C\cap B$ connecting the two opposite faces of $B$.
We define the diameter of a finite cluster $\sC$ as 
$$\Diam(\sC):=\max_{\substack{i=1,\dots, d\\ x,y\in \sC}}|x_i-y_i|\, $$
where $|.|$ represents the standard absolute value.
Let $T_{m,t}(p)$ be the event that $B_t$ has a $p$-crossing cluster and contains some other $p$-open cluster $D$ having diameter at least $m$. 
We say that $B_t(u)$ has a $p$-disjoint property if there exist two disconnected $p$-open clusters in $\bar{B}_t(u)$, both with vertices in $B_t(u)$ and in the boundary of $\bar{B}_t(u)$. We say that $B_t(u)$ has a $p$-blocked property if there is a $p$-open cluster $C$ in $\bar{B}_t(u)$ with vertices in $B_t(u)$ and in the boundary of $\bar{B}_t(u)$, but without vertices in a $t$-cube of $\bar{B}_t(u)$. We say that a $p$-atypical event occurs in $B_t(u)$ if it has a $p$-blocked property or a $p$-disjoint property (see Figure \ref{atypical}).
\begin{proof}[Proof of Lemma \ref{Grim}]
First, note that if  $B_t(u)$ has a $p$-disjoint property and $\bar{B}_t(u)$ has a $p$-crossing cluster, then one of the two disjoint cluster is different from the $p$-crossing cluster. Therefore, there is a $p$-open cluster of diameter greater than $t$ different from the $p$-crossing cluster, so the event $T_{t,3t}(p)$ occurs in the box $\bar{B}_t(u)$. Similarly, let us assume that $B_t(u)$ has a $p$-blocked property and $\bar{B}_t(u)$ and all of its sub-boxes (\textit{i.e}, boxes $B_t(v)$ such that $B_t(v)\subset\bar{B}_t(u)$) have a $p$-crossing cluster. We denote by $C$ the $p$-open cluster in the definition of the $p$-blocked property. Thus, there is at least one cluster among $C$ and the $p$-crossing clusters of the sub-boxes that are disjoint from the $p$-crossing cluster of $\bar{B}_t(u)$ and so the event $T_{t,3t}(p)$ occurs in the box $\bar{B}_t(u)$. Thus,
\begin{align}\label{eqlem1}
\Prb&[\text{a $p$-atypical event occurs in $B_t(u)$ }]\leq \Prb[\bar{B}_t(u)\text{ does not have a $p$-crossing cluster}]\nonumber\\
&\hspace{2.7cm}+3^d\Prb[B_t(u)\text{ does not have a $p$-crossing cluster}]+ \Prb\left[T_{t,3t}(p)\right]\
\end{align}
As the event $\{B_t(u)\text{ doesn't have a $p$-crossing cluster}\}$ is non-increasing in $p$, we have
$$\Prb[B_t(u)\text{ doesn't have a $p$-crossing cluster}]\leq \Prb[B_t(u)\text{ doesn't have a $p_0$-crossing cluster}]\,.$$
The probability for a box $B_t(u)$ not to have a $p_0$-crossing cluster is decaying exponentially fast with $t^{d-1}$, see for instance Theorem 7.68 in \cite{Grimmett99}. Therefore, there exist positive constants $c_1(p_0)$ and $c_2(p_0)$ such that 
\begin{align}\label{eqlem2}
\Prb[B_t(u)\text{ does not have a $p$-crossing cluster}]&\leq c_1(p_0)\exp(-c_2(p_0)t^{d-1})\,.
\end{align}
It remains to prove that there exist positive constants $\kappa(p_0)$ and $\mu(p_0)$ depending only on $p_0$ such that for all $p\geq p_0$, for all positive integers $m$ and $N$
\begin{align}\label{eqlem3}
\Prb[T_{m,N}(p)]\leq\kappa N^{2d}\exp(-\mu m)\,.
\end{align} 
In dimension $d\geq 3$, we refer to the proof of Lemma 7.104 in \cite{Grimmett99}. The proof of Lemma 7.104 requires the proof of Lemma 7.78. The probability controlled in Lemma 7.78 is clearly non decreasing in the parameter $p$. Thus, if we choose $\delta(p_0)$ and $L(p_0)$ as in the proof of Lemma 7.78 for $p_0>p_c(d)$, then these parameters can be kept unchanged for some $p\geq p_0$. Thanks to Lemma 7.104, we obtain 
\begin{align*}
\forall p \geq p_0,\, \Prb(T_{m,N}(p))&\leq d(2N+1)^{2d} \exp\left(\left(\frac{m}{L(p_0)+1}-1\right)\log(1-\delta(p_0))\right)\\
&\leq \frac{d.3^d}{1-\delta(p_0)}N^{2d}\exp\left(-\frac{-\log(1-\delta(p_0))}{L(p_0)+1} m\right)\, .
\end{align*}
We get the result with $$\kappa= \frac{d.3^d}{1-\delta(p_0)}\quad\text{and}\quad\mu=\frac{-\log(1-\delta(p_0))}{L(p_0)+1}>0\,.$$

\noindent In dimension 2, the result is obtained by Couronn{\'e} and Messikh in the more general setting of FK-percolation, see Theorem 9 in \cite{COURONNE200481}. We proceed similarly as in dimension $d\geq3$, the constant appearing in this theorem first appeared in Proposition 6. The probability of the event considered in this proposition is clearly increasing in the parameter of the underlying percolation which have parameter $1-p$, it is an event for the subcritical regime of the Bernoulli percolation. Let us fix a $p_0>p_c(2)=1/2$, then $1-p_0<p_c(2)$ and we can choose the parameter $c(1-p_0)$ and keep it unchanged for some $1-p\leq 1-p_0$. In Theorem 9, we get the expected result with $c(1-p_0)$ for a $p\geq p_0$ and $g(n)=n$. 

Finally, combining inequalities \eqref{eqlem1}, \eqref{eqlem2} and \eqref{eqlem3}, we get
\begin{align*}
\Prb&[\text{a $p$-atypical event occurs in $B_t(u)$}]\\
&\leq c_1(p_0)\exp(-c_2(p_0)(3t)^{d-1})+3^dc_1(p_0)\exp(-c_2(p_0)t^{d-1})+\kappa(p_0)(3t)^{2d}\exp(-\mu(p_0) t)\,.
\end{align*}
The result follows.
\end{proof}
\noindent We have now the key ingredients to prove that the map $p\mapsto \beta_p$ is Lipschitz continuous.
\begin{proof}[Proof of Theorem \ref{heartflow}] Let $p_c< p_0<p_1<1$,$\vv\in\sS^{d-1}$, and $p,q$ such that $p_0\leq p<q\leq p_1$. First, we fix a cube $B(n,\vv)$ and we couple the percolations of parameters $p$ and $q$ in the standard way, \textit{i.e.}, we consider the i.i.d. family $(U(e))_{e\in \E ^d}$ distributed according to the uniform law on $[0,1]$ and we say that an edge $e$ is $p$-open (resp. $q$-open) if $U(e)\geq p$ (resp. $U(e)\geq q $). Thanks to this coupling, we easily obtain that $\tau_p(\vv,n)\leq \tau_q(\vv,n)$ and by dividing by $(2n)^{d-1}$, taking the expectation and letting $n$ go to infinity we conclude that
 
\begin{align}\label{ineqn1}
\beta_p(\vv)\leq \beta_q(\vv)\,.
\end{align}
Let $E_{n,p}$ be a random cutset of minimal size that achieves the minimum in the definition of $\tau_p(n,\vv)$. We consider now another coupling. The idea is to introduce a coupling of the percolations of parameter $p$ and $q$ such that if an edge is $p$-open then it is $q$-open and $E_{n,p}$ is independent of the $q$-state of any edge. Unfortunately, we cannot find such a coupling but we can introduce a coupling that almost has this property. To do so, for each edge we consider two independent Bernoulli random variables $U$ and $V$ of parameters $p$ and $(q-p)/(1-p)$. We say that an edge $e$ is $p$-open if $U(e)=1$ and that it is $q$-open if $U(e)=1$ or $V(e)=1$. Indeed,
$$\Prb[\{U=1\}\cup\{V=1\}]=p+(1-p) \frac{q-p}{1-p}=q\, .$$
 Let $\delta>0$.  
We have,
\begin{align}\label{ineq1}
\Prb&\left[\tau_q(n,\vv)>\tau_p(n,\vv)+\left(\dfrac{q-p}{1-p}+\delta\right)\alpha n^{d-1},\,\cN_{n,p}<\alpha n^{d-1} \right]\nonumber\\
&\hspace{0.8cm}\leq \Prb\left[\tau_q(n,\vv)-\tau_p(n,\vv)>\left(\dfrac{q-p}{1-p}+\delta\right)|E_{n,p}|\right]\nonumber\\
&\hspace{0.8cm}\leq \sum_{\cE}\Prb\left[E_{n,p}=\cE,\, \#\{e\in\cE:(U(e),V(e))=(0,1)\}>\left(\dfrac{q-p}{1-p}+\delta\right)|\cE|\right]\nonumber\\
&\hspace{0.8cm}\leq \sum_{\cE}\Prb[E_{n,p}=\cE]\,\Prb\left[ \#\{e\in\cE:V(e)=1\}>\left(\dfrac{q-p}{1-p}+\delta\right)|\cE|\right]\nonumber\\
&\hspace{0.8cm}\leq  \exp(-2\delta^ 2n^{d-1})
\end{align}
where the sum is over sets $\cE$ that cut $C'_1(n S(\vv),n)$ from $C'_2(n S(\vv),n)$ in $B(n,\vv)$ and where we use in the last inequality Chernoff bound and the fact that $|E_{n,p}|\geq n^{d-1}$ (uniformly in $\vv$). Finally, using inequality \eqref{ineq1} and Theorem \ref{Zhang}, we get
\begin{align*}
\E[\tau_q(n,\vv)]&\leq \E[\tau_q(n,\vv)\ind_{\cN_{n,p}<\alpha n^{d-1} }]+\E[\tau_q(n,\vv)\ind_{\cN_{n,p}\geq\alpha n^{d-1} }]\\
&\leq \E[\tau_p(n,\vv)]+\left(\dfrac{q-p}{1-p}+\delta\right)\alpha n^{d-1}+ |B(n,\vv)|\left(\e^{-2\delta^ 2n^{d-1}}+C_1\e^{-C_2n^{d-1}}\right)\\
&\leq \E[\tau_p(n,\vv)]+\left(\frac{q-p}{1-p}+\delta\right)\alpha n^{d-1}+C_d(2n)^{d} \left(\e^{-2 \delta^2 n^{d-1}}+C_1\e^{-C_2 n^{d-1}}\right)\,,
\end{align*}
where $C_d$ is a constant depending only on $d$.
Dividing by $(2n)^{d-1}$ and by letting $n$ go to infinity, we obtain
\begin{align}
\beta_q(\vv)\leq \beta_p(\vv)+\left(\dfrac{q-p}{1-p}+\delta\right)\frac{\alpha}{2^{d-1}}
\end{align}
and by letting $\delta$ go to $0$,  
\begin{align}\label{inq2}
\beta_q(\vv)\leq \beta_p(\vv)+\kappa(q-p)
\end{align}
where $\kappa=\alpha/((1-p_1)2^{d-1})$.
Combining inequalities \eqref{ineqn1} and \eqref{inq2}, we obtain that 
\begin{align*}
\sup_{\vv \in \sS^{d-1}} |\beta_q(\vv)-\beta_p(\vv)|\leq \kappa |q-p|\,.
\end{align*}
\end{proof}
 
\section{Proof of Theorems \ref{Cheethmd} and \ref{Wulffthmd}}\label{s4}
\begin{proof}[Proof of Theorem \ref{Cheethmd}]
Let $p_c<p_0<p_1<1$ and $p,q\in[p_0,p_1]$. 
We recall that $W_p$ denotes the Wulff crystal for the norm $\beta_p$ such that $\cL^d(W_p)=1/\theta_p$. In this section we aim to prove that the map $p\mapsto \cI_p(W_p)$ is Lipschitz continuous on $[p_0,p_1]$.

  Notice that as the map $p\mapsto\theta_p$ is non-decreasing, for $p<q$ we have 
\begin{align}\label{volume} 
 \cL^d(W_p)\geq \cL^d(W_q)\,.
 \end{align}
 Moreover, the map $p\mapsto\theta_p$ is infinitely differentiable, see for instance Theorem 8.92 in \cite{Grimmett99}. Therefore, there exists a constant $L$ depending on $p_0$, $p_1$ and $d$ such that for all $p,q\in[p_0,p_1]$,
 \begin{align}\label{lip}
 |\theta_p-\theta_q|\leq L|q-p|\,.
 \end{align}
Let us compute now some useful inequalities.
For any set $E\subset\sR^d$ with Lipschitz boundary, by Theorem \ref{heartflow}, we have
\begin{align}\label{ii}
|\cI_p(E)-\cI_q(E)|&=\left|\int_{\partial E}\left(\beta_p(n_E(x))-\beta_q(n_E(x))\right)\cH^{d-1}(dx)\right|\nonumber\\
&\leq \int_{\partial E}\left|\beta_p(n_E(x))-\beta_q(n_E(x))\right|\cH^{d-1}(dx)\leq \kappa|q-p|\cH^{d-1}(\partial E)\,.
\end{align}
We recall that the map $p\rightarrow\beta_p$ is uniformly continuous on $[p_0,p_1]$. We denote by $\beta^{min}$ and $\beta^{max}$ its minimal and maximal value, i.e., for all $\vv\in\sS^{d-1}$ and $p\in[p_0,p_1]$, we have
$$\beta^{min}\leq \beta_p(\vv)\leq \beta^{max}\,.$$
Together with inequality \eqref{volume} and the fact that the Wulff crystal is a minimizer for an isoperimetric problem, we get
\begin{align}\label{i1}
\cI_p(W_p)\leq \cI_p(W_{p_0})= \int_{\partial W_{p_0}}\beta_p(n_{W_{p_0}}(x))\cH^{d-1}(dx)\leq \beta^{max}\cH^{d-1}(\partial W_{p_0})\,.
\end{align}
We also have
\begin{align*}
\cH^{d-1}(\partial W_p)=\int_{\partial W_p}\cH^{d-1}(dx)\leq  \int_{\partial W_p} \frac{\beta_p(n_{W_p}(x))}{\beta^{min}}\cH^{d-1}(dx)\leq \frac{\cI_p(W_p)}{\beta^{min}}\
\end{align*}
and so together with inequality \eqref{i1}, we get
\begin{align}\label{i2}
\cH^{d-1}(\partial W_p)\leq \cH^{d-1}(\partial W_{p_0})\frac{\beta^{max}}{\beta^{min}}\,.
\end{align}
Finally, we obtain combining inequalities \eqref{volume}, \eqref{ii} and \eqref{i2},
\begin{align}\label{f1}
\cI_p(W_p)\geq \cI_q(W_p)-\kappa|q-p|\cH^{d-1}(\partial W_p)\geq  \cI_q(W_q)-\kappa|q-p|\cH^{d-1}(\partial W_{p_0})\frac{\beta^{max}}{\beta^{min}}\,.
\end{align}
As $\cL^d(W_p)=\cL^d(W_q)\cdot \theta_q/\theta_p=\cL^d(W_q (\theta_q/\theta_p)^{1/d})$ and as $W_p$ is the minimizer for the isoperimetric problem associated with the norm $\beta_p$, we have
$$\cI_p(W_p)\leq \cI_p\left(\left(\frac{\theta_q}{\theta_p}\right) ^{1/d}W_q\right)\leq \left(\frac{\theta_q}{\theta_p}\right) ^{(d-1)/d}\cI_p(W_q)\leq\frac{\theta_q}{\theta_p}\cI_p(W_q)$$
and so using inequalities \eqref{lip}, \eqref{ii}, \eqref{i1} and \eqref{i2} 
\begin{align}\label{f2}
\cI_p(W_p)&\leq\frac{\theta_q}{\theta_p}\big(\cI_q( W_q)+\kappa|q-p|\cH^{d-1}(\partial W_q)\big)\nonumber\\
&\leq\left(1+\frac{L}{\theta_{p_0}}|q-p|\right)\left(\cI_q( W_q)+\kappa|q-p|\cH^{d-1}(\partial W_{p_0})\frac{\beta^{max}}{\beta^{min}}\right)\nonumber\\
&\leq \cI_q( W_q)+\beta^{max}\cH^{d-1}(\partial W_ {p_0})\left(\frac{L}{\theta_{p_0}}+\frac{\kappa}{\beta^{min}}\left(1+\frac{L}{\theta_{p_0}}\right)\right)|q-p|\,.
\end{align}
Thus combining inequalities \eqref{f1} and \eqref{f2} together with Theorem \ref{thmheart}, conditionally on the event $\{0\in\sC_{p_0}\}$, we get
\begin{align}
\lim_{n\rightarrow \infty }n|\widehat{\varphi}_n(q)-\widehat{\varphi}_n(p)|=|\cI_p(W_p)-\cI_q(W_q)|\leq \nu|q-p|
\end{align}
where we set $$\nu=\beta^{max}\cH^{d-1}(\partial W_ {p_0})\left(\frac{L}{\theta_{p_0}}+\frac{\kappa}{\beta^{min}}\left(1+\frac{L}{\theta_{p_0}}\right)\right)\,.$$
\end{proof}
 
\begin{proof}[Proof of Theorem \ref{Wulffthmd}]
Let $p_c<p_0<p_1<1$ and $p,q\in [p_0, p_1] $. We consider $\beta_p^*$ the dual norm of $\beta_p$, defined by 
$$\forall x\in \sR^d, \, \beta^*_p(x)=\sup\{x\cdot z\,: \, \beta_p(z)\leq 1\}\,.$$
Then $\beta^*_p$ is a norm. The Wulff crystal $\widehat{W}_{\beta_p}$ associated with $\beta_p$ is in fact the unit ball associated with $\beta^*_p$. Note that the supremum in the definition of $\beta_p^*$ is always achieved for a $z$ such that  $\beta_p(z)= 1$. Let $x\in\sS^{d-1}$. Let $y\in\sS^{d-1}$ be the direction that achieves the supremum for $\beta^*_p(x)$, thus we have 
\begin{align*}
\beta^*_p(x)=x\cdot \frac{y}{\beta_p(y)}
\end{align*}
and so using Theorem \ref{Cheethmd},
\begin{align*}
\beta^*_p(x)- \beta^*_q(x)\leq x\cdot\frac{y}{\beta_p(y)}-x\cdot\frac{y}{\beta_q(y)}&\leq \frac{\|x\|_2\|y\|_2}{\beta_p(y)\beta_q(y)}|\beta_p(y)-\beta_q(y)|\leq \frac{\kappa}{(\beta^{min})^2}|q-p|
\end{align*}
where $\beta^{min}$ was defined in the proof of Theorem \ref{Cheethmd}.
We proceed similarly for $\beta^*_q(x)-\beta^*_p(x)$. Finally, we obtain
\begin{align}
\sup_{x\in\sS^{d-1}}|\beta^*_p(x)- \beta^*_q(x)|\leq \frac{\kappa}{(\beta^{min})^2} |q-p| \,.
\end{align}
We recall the following definition of the Hausdorff distance between two subsets $E$ and $F$ of $\sR^d$:
$$d_\cH(E,F)=\inf \{r\in\sR^+:E\subset F^r\text{ and } F\subset E^r\}$$
where $E^r=\{y:\exists x\in E, \|y-x\|_2\leq r\}$. Thus, we have
$$d_\cH(\widehat{W}_{\beta^*_p},\widehat{W}_{\beta _q})\leq \sup_{y\in\sS^{d-1}}\left\|\frac{y}{\beta^*_p(y)}-\frac{y}{\beta^*_q(y)}\right\|_2\,.$$
Note that $y/\beta^*_p(y)$ (resp. $y/\beta^*_q(y)$) is in the unit sphere for the norm $\beta^*_p$ (resp. $\beta^*_q$). Let $x\in\sS^{d-1}$. Using the definition of $\beta^*$, we obtain
\begin{align*}
\frac{1}{\beta ^{max}}\leq x \cdot \frac{x}{\beta _p(x)}\leq \beta^*_p(x)\,.
\end{align*}
 
Finally, we have
\begin{align}\label{forme}
 d_\cH(\cB_{\beta^*_p},\cB_{\beta^*_q})&\leq \sup_{y\in\sS^{d-1}}\left|\frac{1}{\beta^*_p(y)}-\frac{1}{\beta^*_q(y)}\right|\nonumber\\
 &\leq \sup_{y\in\sS^{d-1}}\frac{1}{\beta^*_q(y)\beta^*_p(y)}\left|\beta^*_p(y)-\beta^*_q(y)\right|\nonumber\\
 &\leq \sup_{y\in\sS^{d-1}}(\beta^{max})^2\left|\beta^*_p(y)-\beta^*_q(y)\right|\leq  \frac{\kappa(\beta^{max})^2}{(\beta^{min})^2} |q-p| \,.
 \end{align}
The result follows.
\end{proof}
\bibliographystyle{plain}
\def\cprime{$'$}

\end{document}

%% file: dessin-blocked_disjoint.pdf_tex
%% Creator: Inkscape inkscape 0.91, www.inkscape.org
%% PDF/EPS/PS + LaTeX output extension by Johan Engelen, 2010
%% Accompanies image file 'dessin-blocked_disjoint.pdf' (pdf, eps, ps)
%%
%% To include the image in your LaTeX document, write
%%   \input{<filename>.pdf_tex}
%%  instead of
%%   \includegraphics{<filename>.pdf}
%% To scale the image, write
%%   \def\svgwidth{<desired width>}
%%   \input{<filename>.pdf_tex}
%%  instead of
%%   \includegraphics[width=<desired width>]{<filename>.pdf}
%%
%% Images with a different path to the parent latex file can
%% be accessed with the `import' package (which may need to be
%% installed) using
%%   \usepackage{import}
%% in the preamble, and then including the image with
%%   \import{<path to file>}{<filename>.pdf_tex}
%% Alternatively, one can specify
%%   \graphicspath{{<path to file>/}}
%% 
%% For more information, please see info/svg-inkscape on CTAN:
%%   http://tug.ctan.org/tex-archive/info/svg-inkscape

\begingroup%
  \makeatletter%
  \providecommand\color[2][]{%
    \errmessage{(Inkscape) Color is used for the text in Inkscape, but the package 'color.sty' is not loaded}%
    \renewcommand\color[2][]{}%
  }%
  \providecommand\transparent[1]{%
    \errmessage{(Inkscape) Transparency is used (non-zero) for the text in Inkscape, but the package 'transparent.sty' is not loaded}%
    \renewcommand\transparent[1]{}%
  }%
  \providecommand\rotatebox[2]{#2}%
  \ifx\svgwidth\undefined%
    \setlength{\unitlength}{489.27436bp}%
    \ifx\svgscale\undefined%
      \relax%
    \else%
      \setlength{\unitlength}{\unitlength * \real{\svgscale}}%
    \fi%
  \else%
    \setlength{\unitlength}{\svgwidth}%
  \fi%
  \global\let\svgwidth\undefined%
  \global\let\svgscale\undefined%
  \makeatother%
  \begin{picture}(1,0.53103551)%
    \put(0,0){\includegraphics[width=\unitlength,page=1]{dessin-blocked_disjoint.pdf}}%
    \put(0.46082863,0.24939112){\color[rgb]{0,0,0}\makebox(0,0)[lt]{\begin{minipage}{0.16350744\unitlength}\raggedright $\bar{B}_t(u)$\end{minipage}}}%
    \put(0.45802717,0.41635821){\color[rgb]{0,0,0}\makebox(0,0)[lt]{\begin{minipage}{0.16350744\unitlength}\raggedright $B_t(u)$\end{minipage}}}%
    \put(0,0){\includegraphics[width=\unitlength,page=2]{dessin-blocked_disjoint.pdf}}%
  \end{picture}%
\endgroup%